\documentclass[11pt]{article}

\usepackage[T1]{fontenc}
\usepackage{lmodern}
\usepackage{amsmath,amssymb,amsthm,mathtools}
\usepackage{microtype}
\usepackage{enumitem}
\usepackage{booktabs}
\usepackage{array}
\usepackage[margin=1.06in]{geometry}
\usepackage[hidelinks]{hyperref}
\hypersetup{
  pdftitle={Short Second Proof of the Odd-Modulus Directed Torus Hamilton Decomposition Theorem},
  pdfauthor={SangHyun Park},
  pdfsubject={A second proof of the odd-modulus directed torus Hamilton decomposition theorem},
  pdfkeywords={directed torus, Hamilton decomposition, cyclic lift, fixed-row-sum selection, voltage}
}

\newtheorem{theorem}{Theorem}[section]
\newtheorem{lemma}[theorem]{Lemma}

\newtheorem{proposition}[theorem]{Proposition}
\theoremstyle{definition}
\newtheorem{definition}[theorem]{Definition}
\newtheorem{example}[theorem]{Example}
\theoremstyle{remark}
\newtheorem{remark}[theorem]{Remark}

\newcommand{\U}{(\mathbb Z/m\mathbb Z)^\times}
\newcommand{\Cay}{\operatorname{Cay}}
\newcommand{\code}[1]{\texttt{#1}}
\newcommand{\eps}{\varepsilon}

\title{Short Second Proof of the Odd-Modulus Directed Torus Hamilton Decomposition Theorem}
\author{SangHyun Park, Yonsei University}
\date{June 20, 2026}

\begin{document}
\maketitle

\begin{abstract}
Let
\[
D_d(m)=\Cay\bigl((\mathbb Z/m\mathbb Z)^d,\{e_1,\ldots,e_d\}\bigr).
\]
We give a second proof that $D_d(m)$ decomposes into $d$ directed Hamilton
cycles for every $d\ge2$ and every odd $m\ge3$.  The main combinatorial input
is a fixed-row-sum selection theorem: if each indexed support $A$ is repeated
in $m$ rows, one can select $\lfloor |A|/2\rfloor$ entries in every row so that
every column total is a unit modulo $m$.  These selections become the voltages
of a cyclic lift splitting one coordinate direction.  In fibre coordinates,
\[
\widehat h_j(x,z)=\bigl(h_j(x),z+\mathbf 1_{\{j\in M(x)\}}\bigr);
\]
a unit total carry makes each lifted factor Hamilton, while the new fibres
retain the direction profile needed for the next split.  Iterating from a
directed $m$-cycle with $d$ parallel arc copies yields $D_d(m)$.
\end{abstract}



\section{Introduction}\label{sec:introduction}

For integers $m\ge3$ and $d\ge2$, let
\[
D_d(m)=\Cay\bigl((\mathbb Z/m\mathbb Z)^d,\{e_1,\ldots,e_d\}\bigr),
\]
the Cartesian product of $d$ consistently oriented $m$-cycles.  Each vertex
has one outgoing arc in every positive coordinate direction, so a Hamilton
decomposition consists of $d$ directed Hamilton cycles partitioning all
$dm^d$ arcs.

\begin{theorem}\label{thm:main}
For every odd integer $m\ge3$ and every integer $d\ge2$, the arc set of
$D_d(m)$ decomposes into $d$ directed Hamilton cycles.
\end{theorem}

Two statements drive the proof.  Theorem~\ref{thm:selection} selects half of
each repeated support in every row so that all column totals are units modulo
$m$.  Theorem~\ref{thm:splitting} uses those totals as voltages to split one
parallel direction while preserving the repeated-row structure.  Indeed, if
$h_j$ is the successor permutation of a Hamilton factor and $M(x)$ is the set
of marked factors at $x$, then the lifted successor is
\begin{equation}\label{eq:intro-skew}
\widehat h_j(x,z)=\bigl(h_j(x),z+\delta_j(x)\bigr),
\qquad
\delta_j(x)=\mathbf 1_{\{j\in M(x)\}}.
\end{equation}
Writing $N$ for the number of base vertices gives
\[
\widehat h_j^{\,N}(x,z)
 =\left(x,z+\sum_v\delta_j(v)\right).
\]
The selected column sum is a unit, hence the fibre return is one $m$-cycle;
the same fibres have constant direction profiles and supply the repeated rows
for the next split.  Starting from a directed $m$-cycle with $d$ parallel arc
copies, repeated splitting reaches $D_d(m)$.

The all-dimensional odd-modulus theorem was first proved in~\cite{ParkAll},
building on the dimension-five and dimension-seven constructions
\cite{ParkD5,ParkD7}.  That proof combines the base dimensions $2,3,5,7$ with
prefix-count and modular-trade constructions, Cartesian product closure, and
the successor rule $b\mapsto2b+1$.  The present note replaces this dimension
synthesis by the single closure operation above.

Hamiltonicity and Hamilton decomposition are distinct: one spanning cycle
need not extend to an arc decomposition.  Hamilton cycles and paths in
products of directed cycles were studied in
\cite{TrotterErdos1978,CurranWitte1985,WitteGallian1984,CurranGallian1996,LanelEtAl2019},
and arc-disjoint Hamilton path packings in~\cite{DarijaniEtAl2022}.  Keating
\cite{Keating1985} treated decompositions of products of two directed cycles;
Meng and Huang~\cite{MengHuang1997} studied decompositions of finite abelian
Cayley digraphs; and Bogdanowicz~\cite{Bogdanowicz2017,Bogdanowicz2020}
studied equal-length decompositions and Hamilton cycles in directed-cycle
products.  Equal-length cycles need not be spanning, and results for
undirected or inverse-closed Cayley graphs do not directly apply to the
positive-basis orientation.  Our lifts are voltage-graph lifts in the sense
of Gross and Tucker~\cite{GrossTucker1987}.

The selection theorem also differs from prescribed-degree modulo-factor
results.  For example, Hasanvand~\cite{Hasanvand2022} realizes fixed degree
residues under high edge-connectivity hypotheses.  Here the incidence graph
may be disconnected, the residues need only lie in the unit group, and the
$m$ identical rows of each support replace global connectivity.

Section~\ref{sec:selection} proves the selection theorem and its even-modulus
obstruction; Section~\ref{sec:splitting} proves cyclic splitting; and
Section~\ref{sec:iteration} completes the iteration.  The appendices record
the Lean formalization and the AI-assisted discovery of the argument.

\section{Fixed-row-sum unit selection}\label{sec:selection}

Throughout this section $m\ge3$ is odd, and residues are taken in
$\mathbb Z/m\mathbb Z$.

\begin{theorem}[Fixed-row-sum unit selection]\label{thm:selection}
Let $I$ and $Y$ be finite sets, and let $(A_\alpha)_{\alpha\in I}$ be an
indexed family of subsets of $Y$, with equal supports at different indices
treated as distinct blocks.  Assume
\[
a_\alpha:=|A_\alpha|\ge2\qquad(\alpha\in I)
\]
and that every $y\in Y$ belongs to at least one $A_\alpha$.  For each
$\alpha$, let $P_\alpha$ be a set of size $m$ and put
$c_\alpha=\lfloor a_\alpha/2\rfloor$.  Then there are maps
\[
F_\alpha:P_\alpha\longrightarrow\binom{A_\alpha}{c_\alpha}
\qquad(\alpha\in I)
\]
such that, for every $y\in Y$,
\begin{equation}\label{eq:global-column}
\gcd\!\left(
\sum_{\alpha\in I}
 \bigl|\{x\in P_\alpha:y\in F_\alpha(x)\}\bigr|,
 m\right)=1.
\end{equation}
\end{theorem}

\subsection{Unit arithmetic}

\begin{lemma}\label{lem:two-units}
Every element of $\mathbb Z/m\mathbb Z$ is a sum of two units.
\end{lemma}

\begin{proof}
Let $s\in\mathbb Z/m\mathbb Z$.  For each prime divisor $p$ of $m$, choose
$a_p\in\mathbb Z/p\mathbb Z$ outside $\{0,s\bmod p\}$.  Since $p\ge3$, such
a residue exists.  By the Chinese remainder theorem, an integer $u$ may be
chosen with $u\equiv a_p\pmod p$ for every prime $p\mid m$; reduce $u$ modulo
$m$.  Neither $u$ nor $s-u$ is divisible by a prime factor of $m$, and hence
both are units in $\mathbb Z/m\mathbb Z$.
\end{proof}

\begin{lemma}\label{lem:unit-flow}
Let $r\in\U$ and let $t\ge0$.  There exist units $u_1,\ldots,u_t\in\U$ such
that
\[
r-u_1-\cdots-u_t\in\U.
\]
\end{lemma}

\begin{proof}
For $t=0$ the assertion is immediate.  If $t$ is odd, write $r=u_1+v$ with
$u_1,v$ units by Lemma~\ref{lem:two-units}, and choose the remaining units in
pairs $1,-1$.  The final difference is $v$.  If $t\ge2$ is even, set
$u_1=1$, write $r-1=u_2+v$ with $u_2,v$ units, and again choose the remaining
units in pairs $1,-1$.
\end{proof}

\subsection{One repeated support}

Let $A$ be a finite set of size $a\ge2$, let $P$ be a set of size $m$, and
put $c=\lfloor a/2\rfloor$.  A \emph{$c$-selection} is a map
\[
F:P\longrightarrow\binom{A}{c}.
\]
Its column residue at $y\in A$ is
\[
\sigma_F(y)=\bigl|\{x\in P:y\in F(x)\}\bigr|\pmod m.
\]
For $y\in A$, write $\eps_y$ for the corresponding standard basis vector of
$(\mathbb Z/m\mathbb Z)^A$.

\begin{lemma}[Local residue vectors]\label{lem:local}
The following $c$-selections exist.
\begin{enumerate}[label=\textup{(\roman*)},leftmargin=2.5em]
\item A selection with $\sigma_F=0$.
\item For every $S\subseteq A$ with $|S|\ge2$, a selection satisfying
\[
\sigma_F(y)\in\U\quad(y\in S),
\qquad
\sigma_F(y)=0\quad(y\in A\setminus S).
\]
\item For distinct $p,q\in A$ and every $u\in\U$, a selection satisfying
\[
\sigma_F=u(\eps_q-\eps_p).
\]
\end{enumerate}
\end{lemma}

\begin{proof}
For (i), use the same fixed $c$-subset of $A$ in every row.

For (ii), write $s=|S|$ and $h=\lfloor s/2\rfloor$, and enumerate
$P=\{x_0,\ldots,x_{m-1}\}$.  If $s=2h$, partition $S$ into pairs.  On each
pair $\{p,q\}$, select $p$ in row $x_0$ and $q$ in every other row.  The two
column counts are $1$ and $m-1$, while every row receives one element of the
pair.

If $s=2h+1$, choose $p,q,r\in S$.  Select $p$ in row $x_0$, $q$ in row
$x_1$, and $r$ in the remaining $m-2$ rows.  Partition the other $2h-2$
elements of $S$ into pairs and use the preceding construction.  Every row
receives $h$ elements of $S$, and every active column count belongs to
$\{1,m-1,m-2\}$.  All three counts are relatively prime to the odd integer
$m$.

In either parity, add to every row a fixed set of $c-h$ elements of
$A\setminus S$.  Such a set exists because
\[
(a-s)-(c-h)
 =\left\lceil\frac a2\right\rceil
  -\left\lceil\frac s2\right\rceil\ge0.
\]
The added columns have count $m$, and unused columns have count $0$.  This
proves (ii).

For (iii), represent $u$ by an integer in $\{1,\ldots,m-1\}$.  Select $q$ in
exactly $u$ rows and $p$ in the other $m-u$ rows.  In every row add a fixed
$(c-1)$-subset of $A\setminus\{p,q\}$; it exists because $c-1\le a-2$.
The residues at $q$ and $p$ are $u$ and $-u$, respectively, and every other
column residue is zero.
\end{proof}

\begin{example}\label{ex:block}
For $m=5$, $A=\{p,q,r,s\}$, and $c=2$, take
\[
\begin{array}{c|ccccc}
\text{row} & 0&1&2&3&4\\ \hline
F(x_j) & \{p,s\}&\{q,s\}&\{r,s\}&\{r,s\}&\{r,s\}.
\end{array}
\]
The column counts are $(1,1,3,5)$, so the residue vector is
$(1,1,-2,0)$.  The first three entries are units and the filler column
vanishes modulo $5$.
\end{example}

\subsection{Proof of the selection theorem}

\begin{proof}[Proof of Theorem~\ref{thm:selection}]
Form the bipartite incidence graph $\Gamma$ on $I\sqcup Y$, with
$\alpha y$ an edge exactly when $y\in A_\alpha$, and work componentwise.
In one component choose a connected subgraph $\mathcal T$ of minimum size
containing every $Y$-vertex.  It is a tree with no $I$-leaf, since such a leaf
and its edge could be deleted without losing a $Y$-vertex.

Root $\mathcal T$ at $\rho\in Y$.  Each block vertex
$\alpha\in I\cap V(\mathcal T)$ has a parent $p_\alpha\in Y$ and a nonempty
child set
\[
C_\alpha=N_{\mathcal T}(\alpha)\setminus\{p_\alpha\}.
\]
Choose a child block $\alpha_0$ of the root.  The local patterns are assigned
according to the following table.

\begin{center}
\renewcommand{\arraystretch}{1.18}
\begin{tabular}{>{\raggedright\arraybackslash}p{0.28\linewidth}
                >{\centering\arraybackslash}p{0.25\linewidth}
                >{\centering\arraybackslash}p{0.29\linewidth}}
\toprule
block type & contribution to its parent & contribution to each child\\
\midrule
root block $\alpha_0$ & a unit at $\rho$ & a unit\\
branching block, $|C_\alpha|\ge2$ & $0$ & a unit\\
unary block, $C_\alpha=\{z\}$ & $-u$ & $+u$\\
block outside $\mathcal T$ & $0$ & $0$\\
\bottomrule
\end{tabular}
\end{center}

Apply Lemma~\ref{lem:local}(ii) to $\{\rho\}\cup C_{\alpha_0}$ at the
root block and to $C_\alpha$ at every other tree block with at least two
children; use Lemma~\ref{lem:local}(i) outside $\mathcal T$, and leave the
unary tree blocks unassigned.

Process the $Y$-vertices in increasing distance from $\rho$.

\smallskip
\noindent\emph{Induction invariant.}
When a vertex $y$ is reached, its parent block has already supplied a unit
residue $r_y$.  Every other assigned block incident with $y$ contributes
$0$, and the unassigned incident blocks are its unary child blocks, except
that $\alpha_0$ has already been assigned at the root.  Every previously
processed $Y$-vertex has its final unit residue.

\smallskip
At $\rho$, the unit $r_\rho$ comes from $\alpha_0$.  Suppose the invariant
holds at $y$.  Let $\alpha_1,\ldots,\alpha_t$ be its unary child blocks and
let $z_i$ be the unique child of $\alpha_i$.
By Lemma~\ref{lem:unit-flow}, choose units $u_1,\ldots,u_t$ such that
\[
r_y-u_1-\cdots-u_t
\]
is a unit.  On $P_{\alpha_i}$ use the transfer pattern
$u_i(\eps_{z_i}-\eps_y)$ from Lemma~\ref{lem:local}(iii).  The final residue
at $y$ is now a unit, and each $z_i$ receives the unit $u_i$ as its incoming
residue.  Branching child blocks contribute $0$ at $y$ and units at their
children.  Hence the invariant advances to the next level.

Every $Y$-vertex is eventually processed.  Its total selected degree is
congruent modulo $m$ to the unit constructed at that vertex, which proves
\eqref{eq:global-column}.  Repeating the construction in each component
completes the proof.
\end{proof}

\begin{example}[A two-level assembly]\label{ex:tree}
Suppose the relevant incidence tree is
\[
\rho-\alpha_0-\{y_1,y_2\},
\qquad
y_2-\alpha_1-z,
\]
where $\alpha_1$ is unary.  The root-block pattern supplies units
$r_\rho,r_{y_1},r_{y_2}$.  Assigning
\[
u(\eps_z-\eps_{y_2})
\]
to $\alpha_1$ changes the final residues to
\[
r_\rho,\qquad r_{y_1},\qquad r_{y_2}-u,\qquad u.
\]
Lemma~\ref{lem:unit-flow} chooses the unit $u$ so that $r_{y_2}-u$ is also a
unit.  Thus the unary block passes a unit to $z$ while leaving a unit at $y_2$,
which is the induction step used in Theorem~\ref{thm:selection}.
\end{example}

\begin{remark}[Why the repeated-support hypothesis matters]\label{rem:replication}
The simple repeated-row structure cannot be replaced by arbitrary parallel
incidences.  For $m=3$, three left vertices with two parallel incidences each
to one right vertex satisfy the naive degree conditions, but selecting one
incidence at each left vertex forces selected degree $3$.  In
Theorem~\ref{thm:selection}, each block instead consists of three simple rows
with a common set of distinct columns, which supports the pair and triple
patterns of Lemma~\ref{lem:local}.
\end{remark}

\begin{proposition}[Failure at even modulus]\label{prop:even-failure}
The conclusion of Theorem~\ref{thm:selection} is false for every even $m$.
\end{proposition}

\begin{proof}
Take one block with support $A=\{1,2,3\}$ and let each of its $m$ rows select
one column.  If the three column counts are $n_1,n_2,n_3$, then
\[
n_1+n_2+n_3=m.
\]
A number relatively prime to even $m$ is odd.  If all three column counts
were units modulo $m$, then the left side would be the sum of three odd
integers and hence odd, whereas $m$ is even.
\end{proof}

\section{Cyclic splitting}\label{sec:splitting}

\subsection{Multitori and fibrewise constancy}

For positive integers $a_1,\ldots,a_r$, define the directed multigraph
$T_m(a_1,\ldots,a_r)$ by
\begin{align*}
V\bigl(T_m(a_1,\ldots,a_r)\bigr)&=(\mathbb Z/m\mathbb Z)^r,\\
E\bigl(T_m(a_1,\ldots,a_r)\bigr)
 &=\{(x,i,s):x\in(\mathbb Z/m\mathbb Z)^r,
       1\le i\le r,\ 1\le s\le a_i\}.
\end{align*}
The labelled arc $(x,i,s)$ has tail $x$ and head $x+e_i$.  Thus
\[
D_d(m)=T_m(\underbrace{1,\ldots,1}_{d\text{ entries}}).
\]

Let $\mathcal H=\{H_1,\ldots,H_d\}$ be a directed Hamilton decomposition of
$T_m(a_1,\ldots,a_r)$, where $d=a_1+\cdots+a_r$.  Write
$h_j:V\to V$ for the successor permutation of $H_j$.  For a vertex $x$ and
a direction $i$, set
\[
A_i(x)=\{j:H_j\text{ uses an }i\text{-direction arc at }x\}.
\]
The outgoing arc partition gives $|A_i(x)|=a_i$.

\begin{definition}\label{def:fibered}
The decomposition $\mathcal H$ is \emph{$m$-fibred} if there are a finite
set $\mathcal B$ and a bijection
\[
\phi:\mathcal B\times\mathbb Z/m\mathbb Z\longrightarrow V
\]
such that, for every $\beta\in\mathcal B$ and every direction $i$, the set
$A_i(\phi(\beta,t))$ is independent of $t$.
\end{definition}

\begin{theorem}[Cyclic splitting]\label{thm:splitting}
Let $m\ge3$ be odd.  Suppose that $T_m(a_1,\ldots,a_r)$ has an $m$-fibred
directed Hamilton decomposition.  If $a_i=a\ge2$, then
\[
T_m\left(a_1,\ldots,a_{i-1},
\left\lceil\frac a2\right\rceil,
\left\lfloor\frac a2\right\rfloor,
 a_{i+1},\ldots,a_r\right)
\]
has an $m$-fibred directed Hamilton decomposition.
\end{theorem}

\begin{lemma}[Unit marking]\label{lem:marking}
Suppose $m$ is odd, $\mathcal H$ is $m$-fibred, and $a_i=a\ge2$.  Put
$c=\lfloor a/2\rfloor$.  There are subsets
\[
M(x)\subseteq A_i(x),\qquad |M(x)|=c\quad(x\in V),
\]
such that for every factor $H_j$,
\begin{equation}\label{eq:kj}
k_j:=|\{x\in V:j\in M(x)\}|
\quad\text{satisfies}\quad
\gcd(k_j,m)=1.
\end{equation}
\end{lemma}

\begin{proof}
Let $\phi:\mathcal B\times\mathbb Z/m\mathbb Z\to V$ witness
Definition~\ref{def:fibered} and put
$A_\beta=A_i(\phi(\beta,t))$, independent of $t$ and of size $a$.
Every factor index lies in some $A_\beta$, since a factor avoiding direction
$i$ would keep the $i$th coordinate constant.  Apply
Theorem~\ref{thm:selection} with columns $\{1,\ldots,d\}$, supports
$A_\beta$, and row sets $\mathbb Z/m\mathbb Z$, and let
$M(\phi(\beta,t))$ be the selected row.  Its column totals are precisely the
numbers $k_j$ in \eqref{eq:kj}.
\end{proof}

\subsection{Construction of the lifted factors}

Fix a direction $i$ with $a_i=a\ge2$, and put
\[
b=\left\lceil\frac a2\right\rceil,
\qquad
c=\left\lfloor\frac a2\right\rfloor.
\]
Let
\[
\widehat T=T_m(a_1,\ldots,a_{i-1},b,c,a_{i+1},\ldots,a_r).
\]
Write a vertex of $\widehat T$ as
\[
\widehat x=(x_1,\ldots,x_{i-1},u,z,x_{i+1},\ldots,x_r)
\]
and define the contraction
\begin{equation}\label{eq:projection}
\pi(\widehat x)
 =(x_1,\ldots,x_{i-1},u+z,x_{i+1},\ldots,x_r).
\end{equation}
The map
\begin{equation}\label{eq:Psi}
\Psi:V(\widehat T)\longrightarrow V(T_m(a_1,\ldots,a_r))\times\mathbb Z/m\mathbb Z,
\qquad
\Psi(\widehat x)=(\pi(\widehat x),z),
\end{equation}
is a bijection; its inverse replaces the old $i$th coordinate $x_i$ by
$(u,z)=(x_i-z,z)$.

Let $M(x)$ be a marking as in Lemma~\ref{lem:marking}, and set
\[
\delta_j(x)=\mathbf 1_{\{j\in M(x)\}}\in\mathbb Z/m\mathbb Z.
\]
At a lifted vertex over $x$, factor $j$ follows the inherited direction used
by $H_j$ if that direction is not $i$.  If $H_j$ uses direction $i$, it uses
the first child direction when $j\notin M(x)$ and the second child direction
when $j\in M(x)$.  For an old direction $\ell\ne i$, denote its new coordinate
index by
\[
\iota_i(\ell)=
\begin{cases}
\ell,&\ell<i,\\
\ell+1,&\ell>i.
\end{cases}
\]

\begin{lemma}[Local arc partition]\label{lem:local-partition}
Parallel-copy labels can be assigned to the lifted factors so that, at every
vertex of $\widehat T$, each outgoing labelled arc is used by exactly one
factor.
\end{lemma}

\begin{proof}
At each base vertex $x$, choose bijections
\[
\lambda_x^-:A_i(x)\setminus M(x)\longrightarrow\{1,\ldots,b\},
\qquad
\lambda_x^+:M(x)\longrightarrow\{1,\ldots,c\}.
\]
These bijections exist because the two sets have sizes $b$ and $c$.  Use
$\lambda_x^-(j)$ as the first-child label for
$j\in A_i(x)\setminus M(x)$ and $\lambda_x^+(j)$ as the second-child label
for $j\in M(x)$.  In an
inherited direction $\iota_i(\ell)$, retain the copy label of the base arc of
$H_j$ at $x$.  Use the same labels at all lifted vertices in $\pi^{-1}(x)$.

The base arc partition handles inherited directions, while the two
bijections handle the child directions; together they exhaust all outgoing
labelled arcs of $\widehat T$.
\end{proof}

Let $\widehat h_j$ be the successor map defined by the lifted arc of factor
$j$.

\begin{lemma}[Skew-product formula]\label{lem:skew-formula}
Under the bijection \eqref{eq:Psi},
\begin{equation}\label{eq:skew-formula}
\Psi\,\widehat h_j\,\Psi^{-1}(x,z)
 =\bigl(h_j(x),z+\delta_j(x)\bigr).
\end{equation}
\end{lemma}

\begin{proof}
If $H_j$ uses an inherited direction, the old vertex changes by the same
coordinate step and $z$ is fixed.  If it uses the first child direction, $u$
increases by one, so $u+z$ increases by one while $z$ is fixed; here
$\delta_j(x)=0$.  If it uses the second child direction, $z$ and $u+z$ both
increase by one; here $\delta_j(x)=1$.  These are exactly the three cases in
\eqref{eq:skew-formula}.
\end{proof}

\begin{lemma}[Cyclic carry]\label{lem:cyclic-carry}
Let $V$ be a finite set, let $h$ be a cyclic permutation of $V$, and let
$\delta:V\to\mathbb Z/m\mathbb Z$.  Define
\[
\widehat h(x,z)=(h(x),z+\delta(x)).
\]
If $s=\sum_{x\in V}\delta(x)$ is a unit modulo $m$, then $\widehat h$ is a
cyclic permutation of $V\times\mathbb Z/m\mathbb Z$.
\end{lemma}

\begin{proof}
Let $N=|V|$.  Since the $h$-orbit of any $x$ lists every vertex once,
\begin{equation}\label{eq:holonomy}
\widehat h^{\,N}(x,z)=(x,z+s).
\end{equation}
Translation by the unit $s$ is one $m$-cycle on the fibre, while between
successive returns the first coordinate visits all $N$ vertices.  Hence the
orbit has $mN$ points and equals $V\times\mathbb Z/m\mathbb Z$.
\end{proof}

\begin{lemma}[Preservation of fibrewise constancy]\label{lem:fiber-preservation}
The lifted decomposition is $m$-fibred, with blocks $\pi^{-1}(x)$ indexed by
old vertices $x$.
\end{lemma}

\begin{proof}
Fix $x$ and vary $\widehat x$ in $\pi^{-1}(x)$.  The factor sets using the
new directions are
\[
\widehat A_{i^-}(\widehat x)=A_i(x)\setminus M(x),
\qquad
\widehat A_{i^+}(\widehat x)=M(x),
\]
where $i^-$ and $i^+$ denote the first and second child directions.  For an
old direction $\ell\ne i$,
\[
\widehat A_{\iota_i(\ell)}(\widehat x)=A_\ell(x).
\]
All these sets depend on $x$ but not on the point of the fibre.  Each fibre
has size $m$ by \eqref{eq:projection}.
\end{proof}

\begin{proof}[Proof of Theorem~\ref{thm:splitting}]
Choose the marking $M(x)$ from Lemma~\ref{lem:marking} and construct the
lifted labelled factors as above.  Lemma~\ref{lem:local-partition} shows that
they partition all arcs of the split multitorus.

For factor $j$, the total carry in Lemma~\ref{lem:cyclic-carry} is
\[
\sum_{x\in V}\delta_j(x)
 =k_j\pmod m.
\]
It is a unit by \eqref{eq:kj}.  Since $h_j$ is a cyclic permutation,
Lemmas~\ref{lem:skew-formula} and~\ref{lem:cyclic-carry} show that
$\widehat h_j$ is a cyclic permutation of the entire lifted vertex set.
Thus every lifted factor is Hamilton.  Lemma~\ref{lem:fiber-preservation}
provides the $m$-fibred structure for the resulting decomposition.
\end{proof}

\begin{example}[The split $T_3(3)\to T_3(2,1)$]\label{ex:split}
The tautological decomposition of $T_3(3)$ has factors $H_1,H_2,H_3$, each
with successor $h_j(x)=x+1$ on $\mathbb Z/3\mathbb Z$.  Mark
\[
M(0)=\{1\},\qquad M(1)=\{2\},\qquad M(2)=\{3\}.
\]
Then $k_1=k_2=k_3=1$.  At each base vertex the marked factor uses the single
second-child arc and the other two factors use the two first-child copies.
For $H_1$ the skew product is
\[
\widehat h_1(x,z)=\bigl(x+1,z+\mathbf 1_{\{x=0\}}\bigr).
\]
Its orbit from $(0,0)$ is
\[
\begin{split}
(0,0)&\to(1,1)\to(2,1)\to(0,1)\to(1,2)\\
     &\to(2,2)\to(0,2)\to(1,0)\to(2,0)\to(0,0),
\end{split}
\]
a Hamilton cycle on the nine lifted vertices.  The other two factors are
cyclic shifts of the same calculation.
\end{example}

\section{Iteration}\label{sec:iteration}

\begin{proof}[Proof of Theorem~\ref{thm:main}]
The multitorus $T_m(d)$ is a directed $m$-cycle with $d$ labelled copies of
each arc.  Following one copy at every vertex gives $d$ Hamilton cycles, and
the resulting decomposition is $m$-fibred with one block.

Apply Theorem~\ref{thm:splitting} whenever a multiplicity $a\ge2$ occurs,
replacing it by
\[
\left(\left\lceil\frac a2\right\rceil,
      \left\lfloor\frac a2\right\rfloor\right).
\]
Each split preserves the sum $d$ and increases the number of positive parts by
one.  Starting from one part, after $d-1$ splits there are $d$ positive parts
summing to $d$, hence all are $1$.  The resulting multitorus is
$T_m(1,\ldots,1)=D_d(m)$, and its lifted factors form the required Hamilton
decomposition.
\end{proof}

\section{Discussion}\label{sec:discussion}

The first proof~\cite{ParkAll} contains independent base constructions and
closure theorems; the present note changes only the route to the common
odd-modulus theorem.

The method is intrinsically odd-modular.  Proposition~\ref{prop:even-failure}
shows that the selection theorem itself fails for even $m$; locally, the
triple pattern also loses the unit $m-2$, and Lemma~\ref{lem:two-units} fails
at the prime $2$.  This does not preclude Hamilton decompositions at even
modulus: $D_3(m)$ has one for every $m\ge3$~\cite{ParkD3}.  A uniform
all-dimensional theorem for even modulus remains open.

\appendix

\section{Lean formalization and reproducibility}\label{app:lean}

A Lean~4 formalization using mathlib~\cite{Lean4,Mathlib2020} is archived in
the public repository~\cite{LeanRepo} at commit
\[
\code{7e96c0e1ed81354fd8cb0cde454b95d47fd6149b},
\]
with Lean and mathlib both pinned to \code{v4.30.0-rc2}.  The declaration
\begin{quote}
\small\nolinkurl{Shared.OddMultitori.oddDirectedTorusGoal_replicated_shortcut_closed}
\end{quote}
formalizes Theorem~\ref{thm:main}: its colour classes partition all labelled
outgoing arcs, and every colour successor is a single cycle on the full
vertex set.

\begin{center}
\small
\renewcommand{\arraystretch}{1.14}
\begin{tabular}{>{\raggedright\arraybackslash}p{0.31\linewidth}
                >{\raggedright\arraybackslash}p{0.59\linewidth}}
\toprule
paper component & principal formal object\\
\midrule
multitori and Hamilton factors &
\nolinkurl{OddMultitori.Vertex}, \nolinkurl{OddMultitori.Arc},
\nolinkurl{OddMultitori.Decomposition}\\
$m$-fibred decomposition &
\nolinkurl{OddMultitori.MFiberedDecomposition}\\
selection theorem & local pair/triple patterns and the rooted component-tree
assignment in \nolinkurl{Shared/OddMultitoriSplit.lean}\\
cyclic carry and splitting &
\nolinkurl{Shared.single_cycle_skewProduct_additive_unit_sum},
\nolinkurl{mFiberedSplittingFromReplicatedBalancingGoal}\\
iteration and the Cayley bridge &
\nolinkurl{splitPartToOnesInContext}, \nolinkurl{standardBridgeGoal}\\
\bottomrule
\end{tabular}
\end{center}

The formal split needs only the uniform-support case of
Theorem~\ref{thm:selection}; the paper states the blockwise varying-support
form.  At the archived commit the focused build and kernel-dependency checks
are reproduced by
\begin{verbatim}
lake env lean Shared/OddMultitoriSplit.lean
lake build Shared

import Shared.OddMultitoriSplit
#print axioms
  Shared.OddMultitori.oddDirectedTorusGoal_replicated_shortcut_closed
\end{verbatim}

\section{AI-assisted discovery and disclosure}\label{app:disclosure}

After the proof in~\cite{ParkAll}, the author used OpenAI GPT-5.5 Pro to seek
a shorter argument based on cyclic lifts.  The system proposed the splitting
idea together with an overgeneral balancing statement; the counterexample in
Remark~\ref{rem:replication} led to the repeated-support hypothesis and the
corrected selection theorem.  OpenAI GPT-5.5 Codex assisted with the Lean
implementation, proof search, finite-set bookkeeping, code review, and
synchronization of the formal endpoint with the manuscript.

GPT-5.5 Pro was used for alternative-proof search, error analysis, development
of the corrected argument, and drafting assistance; GPT-5.5 Codex was used
for formalization support.  The author chose the claims, rejected the false
generalization, checked the mathematical proof and Lean source, selected the
literature, and assumes responsibility for the statements, source, and final
manuscript.  AI tools also assisted copy-editing.

\end{document}